\documentclass{amsart}

\usepackage[utf8]{inputenc}
\usepackage{amsmath, amssymb,amsthm,tikz}
\usepackage{tikz-cd}
\usepackage{tikz}
\usepackage{todonotes}
\usepackage{mathrsfs}
\usepackage{extarrows}
\usepackage{graphicx}
\usepackage{thmtools}
\usepackage{mathtools}
\usepackage{hyperref}

\usepackage[english]{babel}

\usepackage[toc,page]{appendix}

\newcommand{\Z}{\mathbb{Z}}

\newcommand{\Q}{\mathbb{Q}}
\newcommand{\F}{\mathbb{F}}
\newcommand{\A}{\mathbb{A}}
\newcommand{\mcA}{\mathcal{A}}

\newcommand{\G}{\mathbb{G}}

\newcommand{\Oo}{\mathcal{O}}

\newcommand{\an}{\operatorname{an}}

\newcommand{\unr}{\operatorname{unr}}

\newcommand{\rank}{\operatorname{rank}}

\newcommand{\Aut}{\operatorname{Aut}}

\newcommand{\et}{\operatorname{\acute{e}t}}

\newcommand{\dR}{\operatorname{dR}}

\newcommand{\Spec}{\operatorname{Spec}}

\newcommand{\Frac}{\operatorname{Frac}}

\newcommand{\Fil}{\operatorname{Fil}}
\newcommand{\Hom}{\operatorname{Hom}}
\newcommand{\Isom}{\operatorname{Isom}}

\newcommand{\Sh}{\operatorname{Sh}}

\newcommand{\Gal}{\operatorname{Gal}}

\newcommand{\GL}{\operatorname{GL}}

\newcommand{\GSp}{\operatorname{GSp}}

\newcommand{\cris}{\operatorname{cris}}

\hypersetup{
    colorlinks,
    citecolor=blue,
    filecolor=blue,
    linkcolor=blue,
    urlcolor=blue
}

\renewcommand{\mod}{\operatorname{mod}}

\title{Normalization in integral models of Shimura varieties of Hodge type}
\author{Yujie Xu}
\address{Department of Mathematics, Harvard University}
\curraddr{}
\email{yujiex@math.harvard.edu}
\date{}

\numberwithin{equation}{subsection}
\newtheorem{theorem}[equation]{Theorem}
\newtheorem{prop}[equation]{Proposition}

\newtheorem{lem}[equation]{Lemma}
\newtheorem{Coro}[equation]{Corollary}

\theoremstyle{definition}

\newtheorem{remark}[equation]{Remark}
\newtheorem{numberedparagraph}[equation]{}
\newtheorem{conj}[equation]{Conjecture}

\begin{document}

\maketitle

\begin{abstract}
    Let $(G,X)$ be a Shimura datum of Hodge type, and $\mathscr{S}_K(G,X)$ its integral model with hyperspecial (resp.~parahoric, assuming the group is unramified) level structure. We prove that $\mathscr{S}_K(G,X)$ admits a closed embedding, which is compatible with moduli interpretations, into the integral model $\mathscr{S}_{K'}(\GSp,S^{\pm})$ for a Siegel modular variety. In particular, 
    the normalization step in the construction of $\mathscr{S}_K(G,X)$ is redundant. In particular, our results apply to the earlier integral models constructed by Rapoport, Kottwitz etc.~(resp.~Rapoport-Zink etc.), as those models agree with the Hodge type integral models for appropriately chosen Shimura data. 
    
    Moreover, combined with a result of Lan's on the boundary components of toroidal compactifications of integral models, our result also implies that there exist closed embeddings of toroidal compactifications of integral models of Hodge type into toroidal compactifications of Siegel integral models, for suitable choices of cone decompositions. 
\end{abstract}

\tableofcontents

\section{Introduction}
\subsection{Main results and Outline}\label{introduction-main-results-section}
Let $(G,X)$ be a Shimura datum of Hodge type, i.e.~it is equipped with an embedding $(G,X)\hookrightarrow (\GSp(V,\psi),S^{\pm})$, where $V$ is a $\Q$-vector space equipped with a symplectic pairing $\psi$. The embedding of Shimura data induces an embedding of Shimura varieties $\Sh_K(G,X)\hookrightarrow\Sh_{K'}(\GSp,S^{\pm})$, where $K'\subset\GSp(\A_f)$. The moduli interpretation of the Siegel modular variety $\Sh_{K'}(\GSp,S^{\pm})$ naturally gives rise to an integral model $\mathscr{S}_{K'}(\GSp,S^{\pm})$.  
We consider the integral model $\mathscr{S}_K(G,X)$ of $\Sh_K(G,X)$ with hyperspecial (resp.~parahoric) level structure,
as constructed in \cite{Kisin-integral-model} (resp.~\cite{Kisin-Pappas}), which is initially defined as the normalization of the closure of $\Sh_K(G,X)$ inside $\mathscr{S}_{K'}(\GSp,S^{\pm})$. In this article, we show that this construction can be simplified, in that the normalization step is redundant, and that $\mathscr{S}_K(G,X)$ is simply the closure of $\Sh_K(G,X)$ inside $\mathscr{S}_{K'}(\GSp,S^{\pm})$.

Our main theorem is the following, which is independent of the choice of symplectic embeddings.
\begin{theorem}\label{Main-Theorem-intro}
For $K\subset G(\A_f)$ small enough and hyperspecial, there exists some $K'\subset\GSp(\A_f)$, such that we have a closed embedding (``the Hodge embedding'')
\[\mathscr{S}_K(G,X)\hookrightarrow\mathscr{S}_{K'}(\GSp,S^{\pm})\]
Consequentially, the normalization step $\mathscr{S}_K(G,X)\xrightarrow{\nu}\mathscr{S}_K^-(G,X)$ is redundant as the closure $\mathscr{S}_K^-(G,X)$ is already smooth, and the integral model $\mathscr{S}_K(G,X)$ has a moduli interpretation inherited from that of $\mathscr{S}_{K'}(\GSp,S^{\pm})$. 
\end{theorem}
In particular, the Hodge morphism is a closed embedding in the PEL case, where we consider integral models constructed in \cite{Kottwitz} (resp.~\cite{Rapoport-Zink}). 
To see the result in that case (see \cite{PEL-embedding} for details), recall that the Hodge morphism is given by forgetting the $\Oo_B$-action on an abelian scheme $\mcA$, where $B$ is a semisimple $\Q$-algebra attached to the PEL moduli problem. Let $T^{(p)}(\mcA)$ be the prime-to-$p$ Tate module. For a point on $\mathscr{S}_K(G,X)$, the corresponding level structure $\eta:V\otimes \A_f^p\xrightarrow{\sim}T^{(p)}(\mcA)\otimes\A_f^p$ is compatible with the $\Oo_B$-actions. This shows that the $\Oo_B$-action on $\mcA$ is already determined by $\eta$, and hence that the Hodge morphism is an embedding. Strictly speaking, this argument only applies when we let the level structure away from $p$ go to zero, but it is not hard to deduce Theorem \ref{Main-Theorem-intro} from this.

In the general Hodge type case, the mod $p$ points of the integral model $\mathscr{S}_K(G,X)$ can be interpreted as abelian varieties equipped with certain ``mod $p$ Hodge cycles'', which come from reduction mod $p$ of Hodge cycles in characteristic zero.  
We denote the mod $p$ Hodge cycle at a mod $p$ point $x\in\mathscr{S}_K(G,X)$ by a tuple $(s_{\alpha,\ell,x},s_{\alpha,\cris,x})$, which is determined by either its $\ell$-adic \'etale component or its cristalline component (by Proposition \ref{final-num-cris-triviality-motivated}). This is analogous to the case of Hodge cycles in characteristic $0$, which are determined by either their \'etale components or their de Rham components.

More specifically, let $\mathscr{S}_K^-(G,X)$ be the closure of $\Sh_K(G,X)$ in $\mathscr{S}_{K'}(\GSp,S^{\pm})$. By a criterion in \cite{Kisin-mod-p-points} (resp.~\cite{Rong-mod-p}), two mod $p$ points $x,x'\in\mathscr{S}_K(G,X)(k)$ that have the same image 
in $\mathscr{S}_K^-(G,X)(k)$ are equal if and only if $s_{\alpha,\cris,x}=s_{\alpha,\cris,x'}$. Therefore, to show that the normalization morphism is an isomorphism, it reduces to proving the following statement on cohomological tensors: 
\begin{prop}\label{key-tensor-implication}
$s_{\alpha,\ell,x}=s_{\alpha,\ell,x'}\Longrightarrow s_{\alpha,\cris,x}=s_{\alpha,\cris,x'}$.
\end{prop}

By a CM lifting result on $\mathscr{S}_K(G,X)$ due to \cite{Kisin-mod-p-points} (resp.~\cite{Rong-mod-p}), these cohomological tensors lift, up to $G$-isogenies, to Hodge cycles on CM abelian varieties. 
A key observation is that when two mod $p$ points $x,x'\in \mathscr{S}_K(G,X)(k)$ map to the same image in $\mathscr{S}_K^-(G,X)(k)$, they can be CM-lifted using the same torus, whose cocharacter induces the filtration on the Dieudonn\'e modules $\mathbb{D}(\mcA_x)=\mathbb{D}(\mcA_{x'})$ which then identifies the filtrations on the Dieudonn\'e modules associated to CM-liftable mod $p$ points, giving rise to an isogeny \textit{in characteristic zero} between the two CM lifts. This observation allows us to match up the mod $p$ cristalline tensors using the input from $\ell$-adic \'etale tensors, precisely due to the rationality of Hodge cycles in characteristic zero and the existence of an isogeny lift in characteristic zero.

It is worth pointing out that, in the case where the aforementioned cohomological tensors are algebraic--for example, at points where the Hodge conjecture is true--the family of Hodge cycles (tensors) $s_{\alpha}$ that naturally lives over the Hodge type integral model $\mathscr{S}_K(G,X)$ becomes a flat family of algebraic cycles over $\mathscr{S}_K(G,X)$. In this case, $s_{\alpha,\ell,x}=s_{\alpha,\ell,x'}$ implies that the two algebraic cycles corresponding to the two $\ell$-adic cycles are $\ell$-adic cohomologically equivalent, hence numerically equivalent, and we only need to show that they are also cristalline-cohomologically equivalent. Recall that the Grothendieck Standard Conjecture D says that numerical equivalence and cohomological equivalence agree for algebraic cycles. The proof of \ref{key-tensor-implication} thus follows from a cristalline realisation of this Standard Conjecture D, for points on the integral model of Hodge type and their associated cristalline tensors, which are mod $p$ Hodge cycles. 
Our result essentially establishes, unconditionally, rationality for mod $p$ Hodge cycles that live on mod $p$ points of Hodge type Shimura varieties. 

Finally, we state the following two analogues of our Theorem \ref{Main-Theorem-intro}. Firstly, in the case of parahoric integral models constructed in \cite{Kisin-Pappas}, we impose  
mild technical assumptions from  \cite[6.18]{Rong-mod-p}. We certainly expect this technical assumption from \textit{loc.cit.} to be eventually unnecessary in our Theorem below. 
\begin{theorem}\label{parahoric-analogue-intro}
Let $G_{\Q_p}$ be residually split,  
and $K$ a parahoric level structure. There exists a closed embedding $\mathscr{S}_K(G,X)\hookrightarrow\mathscr{S}_{K'}(\GSp,S^{\pm})$ of integral models, for some suitable $K'$. \\
In particular, the normalization step in the construction of $\mathscr{S}_K(G,X)$ is redundant as the closure $\mathscr{S}_K^-(G,X)$ is already normal. 
\end{theorem}
The second analogue concerns toroidal compactifications of integral models of Hodge type constructed in \cite{Keerthi-compactification} (the PEL cases were constructed earlier in \cite{Lan-thesis}). 
Combining our main theorem \ref{Main-Theorem-intro} with an analysis from \cite{Lan-immersion} on the boundary components of toroidal compactifications, one immediately 
obtains the following result. 
\begin{Coro}\label{toroidal-intro}
Let $(G,X)$ be a Shimura datum of Hodge type. 
For each $K\subset G(\A_f)$ 
sufficiently small\footnote{For parahoric level, assume the same condition as in Theorem \ref{parahoric-analogue-intro}.}, there exist collections $\Sigma$ and $\Sigma'$ of cone decompositions, and $K'\subset\GSp(\A_f)$, such that we have a closed embedding of toroidal compactifications of integral models
\[\mathscr{S}_K^{\Sigma}(G,X)\hookrightarrow\mathscr{S}_{K'}^{\Sigma'}(\GSp,S^{\pm})\]
extending the Hodge embedding of integral models. \\
In particular, the normalization step is redundant, and $\mathscr{S}_K^{\Sigma}(G,X)$ can be constructed by simply taking the closure of $\Sh_K(G,X)$ inside $\mathscr{S}_{K'}^{\Sigma'}(\GSp,S^{\pm})$. 
\end{Coro}
The construction of smooth (resp.~normal) integral models of Shimura varieties plays an important part in the Langlands program. For a more detailed historical exposition, see  
\cite{Kisin-integral-model, Kisin-mod-p-points}. On the other hand, such results as Theorems \ref{Main-Theorem-intro}, \ref{parahoric-analogue-intro} and \ref{toroidal-intro} have been useful in various other aspects of arithmetic geometry, e.g.~in the construction of $p$-adic $L$-functions using Euler systems techniques, in the arithmetic intersection theory of special cycles on Shimura varieties and their integral models as in the Kudla-Rapoport program, arithmetic Gan-Gross-Prasad Program etc.

\subsection{Organization} 
In Chapter \ref{Hodge-type-Chapter}, we recall the basic theory of Hodge type integral models as in \cite{Kisin-integral-model,Kisin-mod-p-points} (resp.~\cite{Kisin-Pappas,Rong-mod-p}), and explain how the question of whether the normalization morphism is an isomorphism essentially reduces to proving the key proposition \ref{key-tensor-implication}. 
In section \ref{CM-lifting-integral-models}, we review CM lifting results on the integral model $\mathscr{S}_K(G,X)$, and prove a few lemmas specific to our setting. 
In section \ref{finishing-up-the-proof-section},  we prove the key 
result \ref{key-tensor-implication} (see \ref{final-num-cris-triviality-motivated}), which leads to our Main Theorems \ref{Main-Theorem-intro} and \ref{parahoric-analogue-intro} (see Corollary \ref{coro-normalization-isomorphism}). We give the toroidal compactification version in section \ref{toroidal-cpct-section}.

\textbf{Acknowledgments.} I would like to thank my advisor Mark Kisin for suggesting this problem to me, and for his continued encouragement. 
I would also 
like to thank Ben Howard for an observation that led to a key improvement in the arguments, and 
Keerthi Madapusi Pera for helpful conversations. 
The author is supported by Harvard University graduate student fellowships.

\section{Integral models of Hodge type}\label{Hodge-type-Chapter}
This section reviews the theory of integral models of Hodge type from \cite{Kisin-integral-model,Kisin-mod-p-points}. 
Note that the $p=2$ case is addressed in \cite{2-adic-integral-model}, henceforward we will not emphasize the difference between the $p>2$ and the $p=2$ cases in our expositions. 
Our expositions will mainly focus on the hyperspecial case, and will mention the parahoric analogues from \cite{Kisin-Pappas,Rong-mod-p} whenever necessary for our proofs. 
\subsection{Setup and notations}
\begin{numberedparagraph}\label{Hodge-type-notations-1st}
We fix a $\Q$-vector space $V$ with a perfect alternating pairing $\psi$. For any $\Q$-algebra $R$, denote by $V_R=V\otimes_{\Q}R$. Let $\GSp=\GSp(V,\psi)$ be the corresponding group of symplectic similitudes, and let $S^{\pm}$ be the Siegel double space. We fix an embedding of Shimura data $i:(G,X)\hookrightarrow(\GSp,S^{\pm})$ and assume that $K_p\subset G(\Q_p)$ is hyperspecial, i.e. $K_p=G_{\Z_{(p)}}(\Z_p)$ for some reductive group $G_{\Z_{(p)}}$ over $\Z_{(p)}$ with generic fibre $G$. By \cite[2.3.1, 2.3.2]{Kisin-integral-model} the embedding $i$ of Shimura data is induced by an embedding $G_{\Z_{(p)}}\hookrightarrow \GL(V_{\Z_{(p)}})$ for some $\Z_{(p)}$-lattice $V_{\Z_{(p)}}\subset V$. By Zarhin's trick, up to replacing $V_{\Z_{(p)}}$ by $\Hom_{\Z_{(p)}}(V_{\Z_{(p)}},V_{\Z_{(p)}})^4$, we can assume that $\psi$ also induces a perfect pairing on $V_{\Z_{(p)}}$ which we again denote by $\psi$. For any $\Z_{(p)}$-algebra $R$, we denote $V_R=V_{\Z_{(p)}}\otimes_{\Z_{(p)}}R$. \\
Take $K_p'=\GSp(V_{\Z_{(p)}})(\Z_p)\subset\GSp(\Q_p)$. For each compact open $K^p\subset G(\A_f^p)$, by  \cite[2.1.2]{Kisin-integral-model}, there exists a compact open $K'^p\subset\GSp(\A_f^p)$ such that $K'^p\supset K^p$ and that the embedding $i$ of Shimura data induces an embedding \[\Sh_K(G,X)\hookrightarrow\Sh_{K'}(\GSp,S^{\pm})\]
of $E$-schemes, where $K'=K'_pK'^p$ and $K=K_pK^p$ and $E=E(G,X)$ is the reflex field. 
\end{numberedparagraph}

\begin{numberedparagraph}
Let $\mathcal{B}$ be an abelian scheme over a $\Z_{(p)}$-scheme $T$, and we define the ``prime-to-$p$ Tate module'' to be the \'etale local system on $T$ given by
$\widehat{V}^p(\mathcal{B})=\underset{p\nmid n}{\varprojlim}\mathcal{B}[n]$. 
We denote ``the rational Tate module away from $p$'' by $\widehat{V}^p(\mathcal{B})_{\Q}=\widehat{V}^p(\mathcal{B})\otimes_{\Z}\Q$. We work in the localized category of the category of abelian schemes over $T$, where the morphisms are given by $\Hom$ groups in the usual category tensored with $\Z_{(p)}$. We call an object in this category an \textit{abelian scheme up to prime to $p$ isogeny}. An isomorphism in this category will be called a $p'$-quasi-isogeny.

Let $\mcA$ be an abelian scheme up to prime to $p$ isogeny, and let $\mcA^*$ be the dual abelian scheme. By a \textit{weak polarization} we mean an equivalence class of $p'$-quasi-isogenies $\lambda: \mcA\xrightarrow{\sim}\mcA^*$ such that some multiple of $\lambda$ is a polarization, and two weak polarizations are equivalent if they differ by an element of $\Z_{(p)}^{\times}$.
\end{numberedparagraph}

\begin{numberedparagraph}\label{moduli-integral-model}
For an abelian scheme up to prime to $p$ isogeny with a weak polarization, i.e. a pair $(\mcA,\lambda)$, we denote by $\underline{\Isom}(V_{\A_f^p},\widehat{V}^p(\mcA)_{\Q})$ the \'etale sheaf on $T$ consisting of isomorphisms $V_{\A_f^p}\xrightarrow{\sim}\widehat{V}^p(\mcA)_{\Q}$ which are compatible with the pairings induced by $\psi$ and $\lambda$ up to an $\A_f^{p\times}$-scalar. 
We define a $K'^p$-level structure on $(\mcA,\lambda)$ to be a section 
\begin{equation}\label{K'-p-level-structure}
    \epsilon_{K'}^p\in \Gamma(T,\underline{\Isom}(V_{\A_f^p},\widehat{V}^p(\mcA)_{\Q})/K'^p)
\end{equation}

We consider the following functor from the category $Sch_{/\Z_{(p)}}$ of $\Z_{(p)}$-schemes to the category of sets
\begin{align*}
    Sch_{/\Z_{(p)}}&\to Sets\\
    T&\mapsto \{\text{Isomorphism classes of triples }(\mcA/T, \lambda, \epsilon_{K'}^p)\}
\end{align*}
For $K'^p$ sufficiently small, the above functor is representable by a smooth $\Z_{(p)}$-scheme, which we denote by $\mathscr{S}_{K'}(\GSp,S^{\pm})_{\Z_{(p)}}$, whose $\Q$-fibre is precisely the Siegel modular variety $\Sh_{K'}(\GSp,S^{\pm})$. 

Let $\Oo:=\Oo_E$ denote the ring of integers of the reflex field $E=E(G,X)$. Let $v|p$ be a prime of $E$ lying over $p$, and denote by $\Oo_{(v)}$ the localization at $v$. We denote by
$\mathscr{S}_{K'}(\GSp,S^{\pm})_{\Oo_{(v)}}$ the base change of $\mathscr{S}_{K'}(\GSp,S^{\pm})_{\Z_{(p)}}$ to $\Oo_{(v)}$. Denote by $\mathscr{S}_K^{-}(G,X)$ the closure of $\Sh_K(G,X)$ in $\mathscr{S}_{K'}(\GSp,S^{\pm})_{\Oo_{(v)}}$, and by $\mathscr{S}_K(G,X)$ the normalization of $\mathscr{S}_K^-(G,X)$. We call the normalization $\mathscr{S}_K(G,X)$ the \textit{integral model of Hodge type} for the Hodge type Shimura datum $(G,X)$. In the following, we shall denote by $\nu:\mathscr{S}_K(G,X)\to\mathscr{S}_K^-(G,X)$ the normalization morphism.
\end{numberedparagraph}

\subsection{Cohomological tensors over the integral model}\label{section-cohomological-tensors}
\begin{numberedparagraph}\label{tensor-stablizer}
Recall from \ref{Hodge-type-notations-1st} that the embedding of Shimura data $i$ is induced by a closed embedding 
$G_{\Z_{(p)}}\hookrightarrow\GL(V_{\Z_{(p)}})$
for some $\Z_{(p)}$-lattice $V_{\Z_{(p)}}\subset V$. By \cite[Proposition 1.3.2.]{Kisin-integral-model}, there exists a finite collection of tensors $(s_{\alpha})\subset V_{\Z_{(p)}}^{\otimes}$ such that the subgroup $G_{\Z_{(p)}}\subset \GL(V_{\Z_{(p)}})$ is the scheme-theoretic stabilizer of these tensors $(s_{\alpha})$. 
The moduli interpretation of $\mathscr{S}_{K'}(\GSp,S^{\pm})$ as the moduli space of polarized abelian schemes gives us a universal abelian scheme $\mathscr{A}\to \mathscr{S}_{K'}(\GSp,S^{\pm})$. Therefore, we can pullback this universal abelian scheme to an abelian scheme 
\begin{equation}\label{abelian-scheme-over-integral-model}
h:\mathscr{A}\to\mathscr{S}_K(G,X)
\end{equation}
which, by abuse of notation, we still denote as $\mathscr{A}$. 
Let $V_B=R^1h_*^{\an}\Z_{(p)}$ and $s_{\alpha}$ have Betti realisations $s_{\alpha,B}\in V_B^{\otimes}$. Consider the first relative de Rham cohomology $\mathcal{V}=R^1h_*\Omega^{\bullet}$ of $\mathscr{A}$. The $s_{\alpha}$ can be viewed as parallel sections of the complex analytic vector bundle associated to $\mathcal{V}^{\otimes}$, which lie in the $\Fil^0$ part of the Hodge filtration. By \cite[Propositions 2.2.2, 2.3.9.]{Kisin-integral-model}, these sections have de Rham realisations $s_{\alpha,\dR}\in\mathcal{V}^{\otimes}$ defined over $\Oo_{(v)}$. 

On the other hand, for a prime $\ell\neq p$, consider the \'etale local system $\mathcal{V}_{\ell}$ with $\Q_{\ell}$-coefficients on $\mathscr{S}_K(G,X)$ given by $\mathcal{V}_{\ell}:=R^1{h_{\et}}_*\Q_{\ell}$, where $h_{\et}$ is the map on \'etale sites induced from $h$. The tensors $s_{\alpha}$ have $\ell$-adic \'etale realisations $s_{\alpha,\ell}\in\mathcal{V}_{\ell}^{\otimes}$ which descend to $\Oo_{(v)}$. Moreover, for an $\Oo_{(v)}$-scheme $T$ and $x\in\mathscr{S}_K(G,X)(T)$, we define $s_{\alpha,\ell,x}$ to be the pullback of the section $s_{\alpha,\ell}$ (defined over $\Oo_{(v)}$) to $T$. We denote by $\mcA_x$ the pullback of $\mathscr{A}$ to $x$, thus $s_{\alpha,\ell,x}\in H_{\et}^1(\mcA_{x,\overline{\kappa}},\Q_{\ell})^{\otimes}$, where $\kappa:=\kappa(x)$.

At the prime $p$, we consider the local system $\mathcal{V}_p:=R^1{h_{\eta\et}}_*\Z_p$,
where $h_{\eta}$ the generic fibre of $h$, and the tensors $s_{\alpha}$ have $p$-adic \'etale realisations $s_{\alpha,p}\in \mathcal{V}_{p}^{\otimes}$ which descend to $E$. Likewise, for an $E$-scheme $T$ and $x\in \mathscr{S}_K(G,X)(T)$, we define $s_{\alpha,p,x}$ to be the pullback of $s_{\alpha,p}$ to $T$. Thus $s_{\alpha,p,x}\in H_{\et}^1(\mcA_{x,\overline{\kappa}},\Z_p)^{\otimes}$ where $\kappa:=\kappa(x)$.
\end{numberedparagraph}

\begin{numberedparagraph}\label{promoting-section-paragraph}
For a point $x\in \Sh_K(G,X)(T)\hookrightarrow\Sh_{K'}(\GSp,S^{\pm})(T)$, then the image of $x$ corresponds to a triple $(\mcA_x/T,\lambda,\epsilon_{K'}^p)$, where $\mcA_x$ is an abelian scheme up to prime-to-$p$ isogeny over $T$ and $\epsilon_{K'}^p$ is a section as defined in \ref{K'-p-level-structure}. For any finite index subgroup $K'_1$ such that  $K\subset K_1'\subset K'$, since we have the embedding 
\[\Sh_K(G,X)\hookrightarrow\Sh_{K_1'}(\GSp,S^{\pm})\]
(the image of) $x$ also gives rise to a point of $K_1'$-level structure, i.e. a section $\epsilon_{K_1'}^p\in \Gamma(T,\underline{\Isom}(V_{\A_f^p},\widehat{V}^p(\mcA_x)_{\Q})/{K_1'}^p)$. We define
\begin{equation}\label{promoted-section}
\epsilon_K^p:=\underset{\substack{K_1' \text{s.t.} \\
K\subset K_1'\subset K'}}{\varprojlim}\epsilon_{K_1'}^p\in \Gamma(T,\underline{\Isom}(V_{\A_f^p},\widehat{V}^p(\mcA_x)_{\Q})/K)
\end{equation}
Therefore, given a section $\epsilon_{K'}^p$, we can ``promote'' it to a section $\epsilon_K^p$ as defined in \ref{promoted-section} above, for $K\subset K'$. 
\end{numberedparagraph}

\begin{remark}\label{promoted-level-structure-ell-adic}
The ``promoted'' level structures are realized in terms of the $\ell$-adic tensors, in the sense that
$\epsilon_K^p\colon s_{\alpha}\mapsto (s_{\alpha,\ell})_{\ell\neq p}$.
\end{remark}

\begin{numberedparagraph}\label{fields-notations-integral-model}
We fix an algebraic closure $\overline{\Q}$ of $\Q$. For each place $v$ of $\Q$, we also fix an algebraic closure $\overline{\Q}_v$ of $\Q_v$ and embeddings $\overline{\Q}\hookrightarrow\overline{\Q}_v$.
Let $\overline{\F}_p$ be the residue field of $\overline{\Q}_p$. Denote $L:=\Frac W(\overline{\F}_p)$ and $\Q_p^{\unr}\subset L$ the subfield of elements algebraic over $\Q_p$. We choose a fixed algebraic closure $\overline{L}$ of $L$, and an embedding $\overline{\Q}_p\hookrightarrow \overline{L}$ of $\Q_p^{\unr}$-algebras.  

Recall the notation $E:=E(G,X)$ for the reflex field of $(G,X)$. Let $E_p$ be the completion of $E$ at the prime corresponding to $\overline{\Q}\hookrightarrow\overline{\Q}_p$, and $k\subset\overline{\F}_p$ be a subfield containing the residue field $k_E$ of $E_p$. Let $W:=W(k)$ and $K\subset\overline{L}$ be a finite, totally ramified extension of $W[1/p]$. Take $x\in \mathscr{S}_K(G,X)(k)$ and let $\Tilde{x}\in \Sh_K(G,X)(K)$ be a point that specializes to $x$. Let $\overline{K}$ be the algebraic closure of $K$ in $\overline{L}$. As before, $\mcA_{\Tilde{x}}$ is the pullback of $\mathscr{A}$ to $\Tilde{x}$, and $\mcA_{\Tilde{x},\overline{K}}$ the geometric fibre of $\mathscr{A}$ over $\Tilde{x}$. Thus we have $p$-adic \'etale tensors $s_{\alpha,p,\Tilde{x}}\in H_{\et}^1(\mcA_{\Tilde{x},\overline{K}},\Z_p)^{\otimes}$, which are $\Gal(\overline{K}/K)$-invariant. 
\end{numberedparagraph}
We will need the following result. 
\begin{lem}\label{Kisin-mod-p-Prop-1.3.7.}(\cite[Proposition 1.3.7.]{Kisin-mod-p-points}) (1) Under the $p$-adic comparison isomorphism
\[H_{\et}^1(\mcA_{\widetilde{x},\overline{K}},\Z_p)\otimes_{\Z_p}B_{\cris}\xrightarrow{\sim}H_{\cris}^1(\mcA_x/W)\otimes_WB_{\cris}\]
the $s_{\alpha,p,\Tilde{x}}$ map to $\varphi$-invariant tensors $s_{\alpha,\cris,\Tilde{x}}\in \Fil^0(H_{\cris}^1(\mcA_x/W)^{\otimes})$. \\
(2) There is a $W$-linear isomorphism
\[H_{\et}^1(\mcA_{\Tilde{x},\overline{K}},\Z_p)\otimes_{\Z_p}W\xrightarrow{\sim}H_{\cris}^1(\mcA_x/W)\]
taking $s_{\alpha,p,\Tilde{x}}$ to $s_{\alpha,\cris,\Tilde{x}}$. In particular, the $s_{\alpha,\cris,\Tilde{x}}$ define a reductive group scheme $G_W\subset \GL(H_{\cris}^1(\mcA_x/W))$ which is isomorphic to $G_{\Z_{(p)}}\otimes_{\Z_{(p)}}W$.\\
(3) The filtration on $H_{\cris}(\mcA_x/W)\otimes_Wk$ is given by $\mu_0^{-1}$ where $\mu_0$ is a $G_W$-valued cocharacter conjugate to the Hodge cocharacter $\mu_h$ for $h\in X$. 
\end{lem}

\begin{numberedparagraph}\label{normalization-same-image-point}
Let $x,\widetilde{x}$ be defined as above. Suppose we also have another point $\widetilde{x}'\in \mathscr{S}_K(G,X)(K)$ which specializes to $x'\in \mathscr{S}_K(G,X)(k)$. Suppose moreover that both $x,x'\in\mathscr{S}_K(G,X)(k)$ map to the same image $\overline{x}\in\mathscr{S}_K^-(G,X)(k)$ under the normalization map $\nu$. We shall  make use of the following property about the irreducible components of $\mathscr{S}_K^-(G,X)$. In the hyperspecial case, we have:
\begin{lem}\label{Kisin-Prop-2.3.5}(\cite[Proposition 2.3.5]{Kisin-integral-model})
Let $x\in\mathscr{S}_K^-(G,X)$ be a closed point with characteristic $p$ residue field. Denote by $\widehat{U}_x$ the completion of $\mathscr{S}_K^-(G,X)$ at $x$. Then the irreducible components of $\widehat{U}_x$ are formally smooth over $\Oo_{(v)}$.
\end{lem}
In the parahoric case, the irreducible components are normal. 
\begin{lem}\label{irreducible-components-parahoric}(\cite[Prop.4.2.2]{Kisin-Pappas})
Let $\widehat{U}_{\overline{x}}$ be the completion of $\mathscr{S}_K^-(G,X)_{\Oo_{E^{\unr}}}$ at $\overline{x}$. Then the irreducible component of $\widehat{U}_{\overline{x}}$ containing $x$ is isomorphic to $\widehat{M}_{G,y}^{\mathrm{loc}}$ (which are normal by \cite{Pappas-Zhu}) as formal schemes over $\Oo_{E^{\unr}}$. 
\end{lem}

In \ref{Kisin-Prop-2.3.5} and \ref{irreducible-components-parahoric},  the term ``irreducible components'' of the formal scheme $\hat{U}_x$ refers to the irreducible components of the rigid analytic space attached to $\hat{U}_x$. See \cite[$\mathsection$ 7]{deJong-formal-rigid} for details. 
An immediate corollary of the proof for \ref{Kisin-Prop-2.3.5} (resp.~\ref{irreducible-components-parahoric}) is the following ``cristalline criterion,'' which forms the basis for the proof of main theorem \ref{Main-Theorem-intro}.
\begin{lem}\label{Kisin-1.3.11}(hyperspecial level \cite[Prop 1.3.9., Corollary 1.3.11]{Kisin-mod-p-points}; parahoric level \cite[Corollary 6.3]{Rong-mod-p})\\
(a) $s_{\alpha,\cris,\Tilde{x}}$ depends only on $x$ and not on $\Tilde{x}$. (Thus we will write $s_{\alpha,\cris,x}$ in place of $s_{\alpha,\cris,\Tilde{x}}$ from now on.)\\
(b) Let $x,x'\in\mathscr{S}_K(G,X)(k)$ be two points having the same image in $\mathscr{S}_K^-(G,X)(k)$. Then $x=x'$ if and only if $s_{\alpha,\cris,x}=s_{\alpha,\cris,x'}$.
\end{lem}

We will also need the following result, which goes into the proof of the CM lifting theorem \ref{Kisin-CM-lifting}.
\begin{lem}\label{Kisin-1.1.19}\cite[1.1.19]{Kisin-mod-p-points}
Let $\mu:\G_m\to G$ be a cocharacter, defined over $K$ and conjugate to $\mu_0$. Suppose that $\mu^{-1}$ induces an admissible filtration on $\mathbb{D}(\mathscr{G})_K:=\mathbb{D}(\mathscr{G})(W)\otimes_WK$. 

Then there exists a finite extension $K'/K$ with residue field $k'$, a $p$-divisible group $\widetilde{\mathscr{G}}'$ over $\Oo_{K'}$, and a quasi-isogeny $\theta:\mathscr{G}\to\mathscr{G}'$ where $\mathscr{G}'=\widetilde{\mathscr{G}}'\otimes_{\Oo_{K'}}k'$ such that 
\begin{enumerate}
    \item $\theta$ identifies the filtration on $\mathbb{D}(\mathscr{G}')_K$ corresponding to $\widetilde{\mathscr{G}}'$ and the filtration on $\mathbb{D}(\mathscr{G})_K$ induced by $\mu^{-1}$.
    \item $\theta$ identifies $\mathbb{D}(\mathscr{G}')$ with $g\cdot \mathbb{D}(\mathscr{G})$ for some $g\in G(L)$ with 
    \[g^{-1}b\sigma(g)\in G(\Oo_L)p^{v_0}G(\Oo_L).\]
    \item Viewing $s_{\alpha,\cris}\in \mathbb{D}(\mathscr{G}')^{\otimes}$ via $\theta$, the deformation $\widetilde{\mathscr{G}}'$ of $\mathscr{G}'$ is $gG_{W(k')}g^{-1}$-adapted, where $gG_{W(k')}g^{-1}$ is the stabilizer of $s_{\alpha,\cris}\in\mathbb{D}(\mathscr{G}')^{\otimes}$. 
\end{enumerate}
\end{lem}
The parahoric analogue of this lemma can be extracted from  \cite[$\mathsection$3]{Kisin-Pappas}. 
\end{numberedparagraph}

\subsection{A few more remarks on tensors}\label{remark-on-tensors-section}
\begin{numberedparagraph}
We continue to use the notations as in the previous sections. To emphasize the fact that $\mathscr{S}_K^-(G,X)$ depends on a choice of $K'\subset\GSp(\A_f)$, we denote it, briefly for now, by $\mathscr{S}_{K,K'}^-(G,X)$ instead. 
First we establish the following lemma, see also Corollary \ref{coro-normalization-isomorphism}. 
\begin{lem}\label{same-image-forall-K'}
One of the following two situations always holds: either 
\begin{enumerate}
    \item there exists a sufficiently small $K'\subset\GSp(\A_f)$ such that $\mathscr{S}_K(G,X)\cong \mathscr{S}_{K,K'}^-(G,X)$; or
    \item there exist two points $x,x'\in\mathscr{S}_K(G,X)(k)$ which have the same image $\overline{x}_{K'}\in\mathscr{S}_{K'}(\GSp,S^{\pm})$ for all $K'$ containing $K$.
\end{enumerate}
\end{lem}
\begin{proof}
For each $K'$ containing $K$, let $U_{K'}\subset\mathscr{S}_{K,K'}^-(G,X)$ be the largest open set of points where the normalization $\nu$ is an isomorphism. Let $Z_{K'}$ be its complement, which shrinks when $K'$ shrinks. Since $\mathscr{S}_{K,K'}^-(G,X)$ is Noetherian, the decreasing sequence $\{Z_{K'}\}_{K'}$ stabilizes at some small enough $K_{*}'$. There are then two possibilities: (1) if $Z_{K_*'}=\varnothing$, then $\mathscr{S}_K(G,X)\cong \mathscr{S}^-_{K,K_*'}(G,X)$; (2) if $Z_{K_*'}\neq\varnothing$, then we can take a closed point $\overline{x}\in Z_{K_*'}$, hence $\overline{x}\in Z_{K'}$ for any $K'$ containing $K$, such that $\nu$ is not an isomorphism at $\overline{x}$. Therefore, when we are not in situation (1), there exist two points $x,x'\in\mathscr{S}_K(G,X)$ that map to $\overline{x}\in\mathscr{S}_{K'}(\GSp,S^{\pm})$ for all $K'$ containing $K$. 
\end{proof}

Therefore, to prove Theorem \ref{Main-Theorem-intro}, it suffices to treat only case (2) in \ref{same-image-forall-K'}. We now show that case (2) reduces to an equality of the $\ell$-adic tensors $s_{\alpha,\ell}$. Let $\widetilde{x}\in \mathscr{S}_K(G,X)(K)$ be a characteristic $0$ point that specializes to $x\in\mathscr{S}_K(G,X)(k)$; likewise, let $\widetilde{x}'\in \mathscr{S}_K(G,X)(K)$ be a point that specializes to $x'\in\mathscr{S}_K(G,X)(k)$. 
\begin{lem}\label{ell-adic-tensor-equality}
Fix a level $K\subset G(\A_f)$. If $x,x'\in\mathscr{S}_K(G,X)(k)$ map to the same image point $\overline{x}_{K'}\in\mathscr{S}_{K'}(\GSp,S^{\pm})(k)$ for all $K'$ containing $K$, then $s_{\alpha,\ell,x}=s_{\alpha,\ell,x'}$.   
\end{lem}
\begin{proof}
For each $K'$ containing $K$, by \ref{promoting-section-paragraph} we have sections $\epsilon_{K',\widetilde{x}}$ and $\epsilon_{K',\widetilde{x}'}$ which promote to sections $\epsilon_{K,\widetilde{x}}^p$ and $\epsilon_{K,\widetilde{x}'}^p$ respectively. Since $x$ and $x'$ map to the same image point $\overline{x}_{K'}\in\mathscr{S}_{K'}(\GSp,S^{\pm})$, by \ref{promoted-level-structure-ell-adic}, we have $s_{\alpha,\ell,x}\equiv s_{\alpha,\ell,x'}\mod K'$.
Therefore, $s_{\alpha,\ell,x}\equiv s_{\alpha,\ell,x'}\mod\bigcap\limits_{K\subset K'}K'=K$. Therefore, $s_{\alpha,\ell,x}=s_{\alpha,\ell,x'}$ since they are both invariant under $K$ by construction.
\end{proof}

\begin{remark}
If $x,x'\in\mathscr{S}_K(G,X)(k)$ map to the same image point $\overline{x}\in\mathscr{S}_K^-(G,X)(k)$ under the normalization map $\nu$, and $\widetilde{x},\widetilde{x}'\in\mathscr{S}_K(G,X)(K)$ are points that specialize to $x,x'$ respectively. Then $\mcA_{x}=\mcA_{\overline{x}}=\mcA_{x'}$; moreover, $\mcA_{\widetilde{x}}$ reduce mod $p$ to $\mcA_x$, and $\mcA_{\widetilde{x}'}$ reduce mod $p$ to $\mcA_{x'}$, i.e. $\mcA_{\widetilde{x}}$ and $\mcA_{\widetilde{x}'}$ have the same mod $p$ reduction.
\end{remark}

Now by Lemmas
\ref{Kisin-1.3.11}, \ref{same-image-forall-K'} and \ref{ell-adic-tensor-equality}, to show that the normalization morphism is an isomorphism, it suffices to show that: (i.e. Proposition \ref{key-tensor-implication})
\begin{equation*}
    s_{\alpha,\ell,x}=s_{\alpha,\ell,x'}\Longrightarrow s_{\alpha,\cris,x}=s_{\alpha,\cris,x'}
\end{equation*}

\begin{remark}
If we knew that both the Hodge conjecture for abelian varieties and the Grothendieck standard conjecture D for abelian varieties in characteristic $p$ are true, then our result follows trivially.
\end{remark}
\end{numberedparagraph}

\section{Proof of Main Theorem}

\subsection{CM lifting on the integral model}\label{CM-lifting-integral-models}

\begin{numberedparagraph}
First we recall some CM lifting results on the integral model $\mathscr{S}_K(G,X)$ from \cite[$\mathsection$2]{Kisin-mod-p-points}, whose notations we adopt. Let $v$ be a miniscule cocharacter. Consider the affine Deligne-Lusztig variety
\[X_v(b)=\{g\in G(L)/G(\Oo_L): g^{-1}b\sigma(g)\in G(\Oo_L)p^vG(\Oo_L)\}.\]
For an arbitrary point $x\in\mathscr{S}_K(G,X)(k)$, 
consider the abelian variety $\mcA_x$ and its $p$-divisible group $\mathscr{G}_x$. Consider the Dieudonne module $\mathbb{D}(\mathscr{G}_x):=\mathbb{D}(\mathscr{G}_x)(\Oo_L)$. Since $\mathbb{D}(\mathscr{G}_x)$ is a Dieudonn\'e module, $v$ acting on $\mathbb{D}(\mathscr{G}_x)$ has non-negative weights and induces a minuscule cocharacter of $\GL(\mathbb{D}(\mathscr{G}_x))$. 
If $g\in X_v(b)$, then $g\cdot \mathbb{D}(\mathscr{G}_x)$ is stable under Frobenius and satisfies the axioms of a Dieudonn\'e module. Thus $g\cdot\mathbb{D}(\mathscr{G}_x)$ corresponds to the Dieudonn\'e module of a $p$-divisible group $\mathscr{G}_{gx}$, which is naturally equipped with a quasi-isogeny $\mathscr{G}_x\to\mathscr{G}_{gx}$, corresponding to the natural isomorphism 
$g\cdot \mathbb{D}(\mathscr{G}_x)\otimes_{\Z_p}\Q_p\xrightarrow{\sim}\mathbb{D}(\mathscr{G}_x)\otimes_{\Z_p}\Q_p$. 
Note that by the definition of $G$-invariant tensors, we have
\begin{equation}
s_{\alpha,\cris,x}=g(s_{\alpha,\cris,x})\in (g\mathbb{D}(\mathscr{G}_x))^{\otimes}=\mathbb{D}(\mathscr{G}_{gx})^{\otimes}.
\end{equation}
Denote by $\mcA_{gx}$ the abelian variety corresponding to $\mathscr{G}_{gx}$. It is isogenous to $\mcA_x$ by construction, and comes equipped with a canonical $K'$-level structure $\epsilon_{K',g\cdot x}$ induced from the level structure on $\mcA_x$, and a weak polarization $\lambda_{g\cdot x}$ induced from the weak polarization on $\mcA_{x}$ since $G\subset \GSp$. 
We define a map 
\begin{align}\label{ADLV-map-to-Siegel}
    X_v(b)&\to \mathscr{S}_{K'}(\GSp,S^{\pm})(\overline{\F}_p)\\
    g&\mapsto (\mcA_{gx},\lambda_{gx})
\end{align}

\begin{lem}\label{lifting-isogeny-hyperspecial}\cite[Prop.1.4.4.]{Kisin-mod-p-points}
There is a unique lifting of \ref{ADLV-map-to-Siegel} to a map
\begin{equation}\label{iota-x-1.4.4}
\iota_x:X_v(b)\to\mathscr{S}_K(G,X)(\overline{\F}_p)
\end{equation}
such that $s_{\alpha,\cris,x}=s_{\alpha,\cris,\iota_x(g)}\in \mathbb{D}(\mathscr{G}_{gx})^{\otimes}$. 
\end{lem}
Note that the uniqueness here simply follows from Lemma \ref{Kisin-1.3.11}(b). Moreover, \ref{iota-x-1.4.4} extends to a $\langle\Phi\rangle\times Z_G(\Q_p)\times G(\A_f^p)$-equivariant map
\begin{equation}\label{}
    \iota_x:X_v(b)\times G(\A_f^p)\to\mathscr{S}_{K_p}(G,X)(\overline{\F}_p)
\end{equation}
We call the image of $\iota_x$ the \textit{$G$-isogeny class} of $x$. Let $\delta\in G(K_0)$ be as in \cite[2.1.2]{Kisin-mod-p-points}, where $K_0=W(k)[1/p]$. More specifically, we fix an isomorphism
\[\mathbb{D}(\mathscr{G}_x)\xrightarrow{\sim}V_{\Z_{(p)}}^*\otimes_{\Z_{(p)}}W\]
which takes $s_{\alpha,\cris,x}$ to $s_{\alpha}$. Thus we have identifications 
\[G_{\Z_{(p)}}\otimes_{\Z_{(p)}}K_0=G_{K_0}\xrightarrow{\sim}G(s_{\alpha,\cris,x})\subset \GL(\mathbb{D}(\mathscr{G}_x)_{K_0})\xrightarrow{\sim}\GL(H_{\cris}^1(\mcA_x/W)_{K_0}),\]
where $G(s_{\alpha,\cris,x})\subset\GL(\mathbb{D}(\mathscr{G}_x)_{K_0})$ denotes the subgroup defined by $s_{\alpha,\cris,x}$. Therefore, the Frobenius on $\mathbb{D}(\mathscr{G}_x)$ has the form $\delta\sigma$ with $\delta\in G(K_0)$, where $\sigma$ is the absolute Frobenius on $W(\overline{\F}_p)$. Note that the element $\delta$ is independent of the choices made in its construction, up to $\sigma$-conjugacy by elements of $G(W)$. 
\end{numberedparagraph}

\begin{numberedparagraph}
As in \cite[$\mathsection$ 2.1.2]{Kisin-mod-p-points}, attached to a point $x\in\mathscr{S}_K(G,X)(k)$, let $\gamma_{\ell}\in G(\Q_{\ell})$ be the geometric $q$-Frobenius in $\Gal(\overline{\F}_p/k)$ acting on $H_{\et}^1(\mcA_{\overline{x}},\Q_{\ell})$ and $I_{\ell/k}=I_{\ell,x/k}\subset G_{\Q_{\ell}}$ the centralizer of $\gamma_{\ell}$. Let $I_{\ell,n}$ be the centralizer of $\gamma_{\ell}^n$ in $G_{\Q_{\ell}}$, which forms a decreasing sequence in $G_{\Q_{\ell}}$ and stabilizes to what we shall denote as $I_{\ell}=I_{\ell,x}$ for $n$ sufficiently large. 

Likewise, we can define $I_{p/k}=I_{p,x/k}$ as the algebraic group over $\Q_p$ whose $R$-points for a $\Q_p$-algebra $R$ is given by 
\begin{equation}\label{defining-Ipk}
I_{p/k}(R):=\{\alpha\in G(W\otimes_{\Z_p}R):\delta\sigma(\alpha)=\alpha\delta\}.
\end{equation}
For each positive integer $n$, let $k_n\subset\overline{\F}_p$ be the degree-$n$ extension of $k$ and we define $I_{p,n}$ by replacing $W$ with $W(k_n)$ in \ref{defining-Ipk}. For $n$ sufficiently large, the increasing sequence $I_{p,n}$ (as subgroups inside the group $J_{\delta}$ defined in \cite[1.2.12]{Kisin-mod-p-points}) stabilizes and we denote this group by $I_p:=I_{p,x}$. 
We write 
\begin{equation}\label{defining-Ix}
I_x\subset\Aut_{\Q}(\mcA_{x}\otimes_k\overline{\F}_p)=\Aut_{\Q}(\mcA_{\overline{x}}\otimes_k\overline{\F}_p)
\end{equation}
for the subgroup whose points $I_x(R)$, for a $\Q$-algebra $R$, consist of those elements of $\Aut_{\Q}(\mcA_x\otimes_k\overline{\F}_p)(R)$ fixing the tensors $s_{\alpha,\cris,x}$ and $s_{\alpha,\ell,x}$ for all $\ell\neq p$. Similarly, attached to the point $x'\in\mathscr{S}_K(G,X)(k)$, we define $I_{x'}$ analogously. Let $\varphi$ be the Frobenius. The following Lemma is essentially the analogous result to \cite[Main Theorem]{Tate-endo-ab-var-fin-fields} for the Hodge cycles $s_{\alpha}$. 
\begin{lem}\label{rank-I-equals-rank-Iell}\cite[2.1.7]{Kisin-mod-p-points}
For some prime $\ell\neq p$, $I_{\Q_{\ell}}=I\otimes_{\Q}\Q_{\ell}$ contains the connected component of the identity in $I_{\ell}$. In particular, $\rank I=\rank I_{\ell}=\rank G$. 
\end{lem}

The following CM lifting Theorem (and its proof) is crucial to our proofs. Thus we also sketch its proof. 
\begin{lem}\label{Kisin-CM-lifting}
(\cite[Theorem 2.2.3.]{Kisin-mod-p-points}) The isogeny class $\iota_x(X_v(\delta)\times G(\A_f^p))$ contains a point which is the reduction of a special point on $\Sh_K(G,X)$.
\end{lem}
\begin{proof} (Sketch)
We take a maximal torus $T\subset I_p$ defined over $\Q_p$. Since $I$ and $I_p$ have the same rank, we can assume that $T$ is induced by a maximal torus $T$ in the $\Q$-group $I$, and the induced action of $T$ on $\mathbb{D}(\mathscr{G}_{\widetilde{x}})_K$ respects filtrations, where  $\mathscr{G}_{\widetilde{x}}$ is a $G_W$-adapted deformation such that the filtration on $\mathbb{D}(\mathscr{G}_{\widetilde{x}})_K$ is given by $\mu_T^{-1}$ (by Lemma \ref{Kisin-1.1.19}). One can then check that $T$ is in fact a maximal torus in $G$ and that $\widetilde{x}$ is a special point because the Mumford-Tate group commutes with $T$ and is hence a subgroup of $T$. 
\end{proof}
\end{numberedparagraph}

In the parahoric case, the analogue of Lemma \ref{lifting-isogeny-hyperspecial} is given in \cite[Prop.6.5]{Rong-mod-p} (from whence we inherited the assumption on  $G_{\Q_p}$), where the isogeny classes are parametrized by a certain $X(\sigma\{\mu_y\},b)$, which is a certain union of affine Deligne-Lusztig varieties over certain $\mu$-admissible set. The parahoric version of the CM lifting theorem \ref{Kisin-CM-lifting} is given in \cite[Theorem 9.4]{Rong-mod-p} under the same assumption on $G_{\Q_p}$ as Prop.~6.5 \textit{loc.cit.}

\subsection{Finishing up the proof}\label{finishing-up-the-proof-section}
\begin{numberedparagraph}\label{last-section-recall-CM-lifting}
Recall the setting from \ref{remark-on-tensors-section} that we start with two mod $p$ points $x,x'\in\mathscr{S}_K(G,X)(k)$ on the normalized integral model that map to the same image $\overline{x}\in\mathscr{S}_K^-(G,X)(k)$ for all $K'\supset K$. In particular, we have $\mcA_{x}=\mcA_{x'}=\mcA_{\overline{x}}$ by pulling back the abelian scheme $\mathscr{A}\to\mathscr{S}_K(G,X)$ to the point $x$ or $x'$.

Consider the isogeny class $\iota_{x}(X_v(\delta)\times G(\A_f^p))$. By \ref{Kisin-CM-lifting}, there exists a point $y:=\iota_{x}(g)\in \iota_{x}(X_v(\delta)\times G(\A_f^p))$, for some $g\in X_v(\delta)\times G(\A_f^p)$, such that $\mcA_y$ is $G$-isogenous to $\mcA_x$ and such that $s_{\alpha,\cris,x}=s_{\alpha,\cris,y}\in\mathbb{D}(\mathscr{G}_{gx})^{\otimes}$, and such that $y$ is the reduction of a special point $\widetilde{y}\in \Sh_K(G,X)(K')$.

\begin{lem}\label{Ix=Ix'}
Let $x,x'\in\mathscr{S}_K(G,X)(k)$ be as 
in \ref{last-section-recall-CM-lifting} and $I_x, I_{x'}$ as defined in \ref{defining-Ix}. We have $I_x=I_{x'}$ as subgroups of $\Aut_{\Q}(\mcA_{\overline{x}}\otimes_k\overline{\F}_p)$.
\end{lem}
\begin{proof}
By Lemma \ref{rank-I-equals-rank-Iell}, we have
\[I_x\otimes_{\Q}\Q_{\ell}=\{g\in\Aut_{\Q_{\ell}}(H^1_{\et}(\mcA_{x,\overline{k}},\Q_{\ell}))\big| gs_{\alpha,\ell,x}=s_{\alpha,\ell,x}\text{ for all }\ell\neq p, g\varphi=\varphi g
\},\]
and likewise
\[I_{x'}\otimes_{\Q}\Q_{\ell}=\{g\in\Aut_{\Q_{\ell}}(H^1_{\et}(\mcA_{x',\overline{k}},\Q_{\ell}))\big| gs_{\alpha,\ell,x'}=s_{\alpha,\ell,x'}\text{ for all }\ell\neq p, g\varphi=\varphi g\}.\]
Since $s_{\alpha,\ell,x}=s_{\alpha,\ell,x'}$ for all $\ell\neq p$ by Lemma \ref{ell-adic-tensor-equality}, we have $I_{x}\otimes_{\Q}\Q_{\ell}=I_{x'}\otimes_{\Q}\Q_{\ell}$, thus we have 
$I_x=I_{x'}\subset\Aut_{\Q}(\mcA_{\overline{x}}\otimes_{\Q}\overline{\F}_p)$. 
\end{proof}
\end{numberedparagraph}

\begin{numberedparagraph}
By the proof of Lemma \ref{Kisin-CM-lifting}, to make such a CM lift $\widetilde{y}$ of point $y$, one finds a maximal torus $T_x$
in $I_x$. The induced action of $T_x$ on $\mathbb{D}(\mathscr{G}_{\widetilde{y}})_{K'}$ respects filtrations, and thus the action of $T_x$ on $\mcA_y$ lifts to $\mcA_{\widetilde{y}}$.
\begin{lem}
There exists a special point lift $\widetilde{y}'$ of $x'$ (up to isogeny), such that there is an isogeny $\mcA_{\widetilde{y}}\to \mcA_{\widetilde{y}'}$.
\end{lem}
\begin{proof}
Since $I_x=I_{x'}$ by Lemma \ref{Ix=Ix'},
we can take the same torus $T:=T_x=T_{x'}\subset I_x=I_{x'}$ to construct the CM liftings for $x$ and $x'$, i.e.~$\mcA_{\widetilde{y}}$ and $\mcA_{\widetilde{y}'}$. By Lemma \ref{Kisin-1.1.19}, the filtration on $\mathbb{D}(\mathscr{G}_y)$ corresponding to the deformation $\mathscr{G}_{\widetilde{y}}$ is identified via the isogeny $\mathscr{G}_y\to \mathscr{G}_x$ (induced by $g\in X_v(\delta)$) with the filtration on $\mathbb{D}(\mathscr{G}_x)=\mathbb{D}(\mathscr{G}_{x'})$ induced by $\mu_{T_x}=\mu_{T_{x'}}$, and this filtration is then identified with that on $\mathbb{D}(\mathscr{G}_{y'})$ induced by the deformation $\mathscr{G}_{\widetilde{y}'}$ (of $\mathscr{G}_{y'}$). In particular, the filtration induced by $\mcA_{\widetilde{y}}$ on $\mathbb{D}(\mathscr{G}_y)$ is identified with the filtration induced by $\mcA_{\widetilde{y}'}$ on $\mathbb{D}(\mathscr{G}_{\widetilde{y}'})$, thus we obtain a priori a map between $p$-divisible groups $\mathscr{G}_{\widetilde{y}}\to \mathscr{G}_{\widetilde{y}'}$ whose reduction mod $p$ is precisely the isogeny $\mathscr{G}_y\to\mathscr{G}_{y'}$ given by the composition $\mathscr{G}_y\to\mathscr{G}_x=\mathscr{G}_{x'}\to \mathscr{G}_{y'}$ of $G$-isogenies. (Note that $\mathscr{G}_y\to\mathscr{G}_{y'}$ is not \textit{a priori} a $G$-isogeny.) Since we have a map between the mod $p$ abelian varieties $\mcA_y\to \mcA_{y'}$ via composing the $G$-isogenies to and from $\mcA_x=\mcA_{x'}$, and a lifting of this map for $p$-divisible groups $\mathscr{G}_{\widetilde{y}}\to \mathscr{G}_{\widetilde{y}'}$, by Serre-Tate theory, we then obtain an isogeny between abelian varieties $\mcA_{\widetilde{y}}\to \mcA_{\widetilde{y}'}$. 
\end{proof}

\begin{lem}\label{mapping-tildey-to-tildey'-Hodge-cycles}
The isogeny (but not \textit{a priori} $G$-isogeny) $\mcA_{\widetilde{y}}\to\mcA_{\widetilde{y}'}$ sends the Hodge cycle $(s_{\alpha,\ell,\widetilde{y}},s_{\alpha,\dR,\widetilde{y}})$ to the Hodge cycle $(s_{\alpha,\ell,\widetilde{y}'},s_{\alpha,\dR,\widetilde{y}'})$. 
\end{lem}
\begin{proof}
Under the $G$-isogeny $\mcA_y\to \mcA_x$, the tensors $s_{\alpha,\ell,y}$ get sent to $s_{\alpha,\ell,x}=s_{\alpha,\ell,x'}$, which then get sent to $s_{\alpha,\ell,y'}$ under the $G$-isogeny $\mcA_{x'}\to\mcA_{y'}$. In particular, the isogeny (but not \textit{a priori} a $G$-isogeny) $\mcA_y\to\mcA_{y'}$, which factors through $\mcA_x=\mcA_{x'}$, sends $s_{\alpha,\ell,y}$ to $s_{\alpha,\ell,y'}$. Under the specialization isomorphism $H^*_{\et}(\mcA_{y,\overline{k}},\Q_{\ell})\to H^*_{\et}(\mcA_{\widetilde{y},\overline{K}},\Q_{\ell})$, the isogeny $\mcA_{\widetilde{y}}\to\mcA_{\widetilde{y}'}$ sends $s_{\alpha,\ell,\widetilde{y}}$ to $s_{\alpha,\ell,\widetilde{y}'}$. Since $\mcA_{\widetilde{y}}\to\mcA_{\widetilde{y}'}$ is an isogeny in characteristic zero and Hodge cycles in characteristic zero are determined by either its $\ell$-adic \'etale or de Rham components (as they both come from the Betti realizations), the isogeny $\mcA_{\widetilde{y}}\to\mcA_{\widetilde{y}'}$ (again not \textit{a priori} known to be a $G$-isogeny) sends the Hodge cycle $(s_{\alpha,\ell,\widetilde{y}},s_{\alpha,\dR,\widetilde{y}})$ to $(s_{\alpha,\ell,\widetilde{y}'},s_{\alpha,\dR,\widetilde{y}'})$. 
\end{proof}

\begin{Coro}\label{y-to-y'-cris}
The isogeny (again not \textit{a priori} $G$-isogeny) $\mcA_y\to \mcA_{y'}$ sends $s_{\alpha,\cris,y}$ to $s_{\alpha,\cris,y'}$. 
\end{Coro}
\begin{proof}
By Lemma \ref{mapping-tildey-to-tildey'-Hodge-cycles}, the mod $p$ reduction $\mcA_y\to\mcA_{y'}$ of $\mcA_{\widetilde{y}}\to \mcA_{\widetilde{y}'}$ sends the cristalline realizations $s_{\alpha,\cris,y}$ to $s_{\alpha,\cris,y'}$, via the specialization isomorphism $H^*_{\dR}(\mcA_{\widetilde{y}})\to H^*_{\cris}(\mcA_y/W)\otimes_WK$. 
\end{proof}
Now we are ready to prove the key proposition. 
\begin{prop}\label{final-num-cris-triviality-motivated}
$s_{\alpha,\ell,x}=s_{\alpha,\ell,x'}\Longrightarrow s_{\alpha,\cris,x}=s_{\alpha,\cris,x'}$
\end{prop}
\begin{proof}
Combining Lemmas \ref{Ix=Ix'} through \ref{y-to-y'-cris}, the isogeny $\mcA_y\to\mcA_{y'}$ sends $s_{\alpha,\cris,y}$ to $s_{\alpha\cris,y'}$. On the other hand, this isogeny factors through the $G$-isogeny $\mcA_y\to \mcA_x$ which sends $s_{\alpha,\cris,y}\mapsto s_{\alpha,\cris,x}$ and $G$-isogeny which sends $s_{\alpha,\cris,x'}\mapsto s_{\alpha,\cris,y'}$, thus we must also have $s_{\alpha,\cris,x}=s_{\alpha,\cris,x'}$ (otherwise the image of $s_{\alpha,\cris,y}$ via this composition of isogenies $\mcA_y\to\mcA_x=\mcA_{x'}\to \mcA_{y'}$ would get sent to something other than $s_{\alpha,\cris,y'}$, causing a contradiction). 
\end{proof}
\end{numberedparagraph}

\begin{remark}
The Proposition \ref{final-num-cris-triviality-motivated} essentially suggests that the mod $p$ points of the integral model of Hodge type can be interpreted as abelian varieties equipped with a well-defined notion of ``mod $p$ Hodge cycles,'' written as tuples $(s_{\alpha,\ell,x},s_{\alpha,\cris,x})$, which are determined by either their $\ell$-adic \'etale components or their cristalline components. This is analogous to the case with absolute Hodge cycles in characteristic $0$, which are determined by either their \'etale components or their de Rham components. 
\end{remark}

\begin{Coro}\label{coro-normalization-isomorphism}
The normalization morphism $\nu$ is an isomorphism. In particular, $\mathscr{S}_{K}(G,X)$ admits a closed embedding into $\mathscr{S}_{K'}(\GSp,S^{\pm})$.
\end{Coro}
\begin{proof}
Resume the setting in \ref{remark-on-tensors-section}. For any two mod $p$ points $x,x'\in\mathscr{S}_K(G,X)(k)$ that map to the same image in $\mathscr{S}_K^-(G,X)(k)$ for all $K'$ containing $K$, by Lemma \ref{ell-adic-tensor-equality} and Proposition \ref{final-num-cris-triviality-motivated}, we have $s_{\alpha,\cris,x}=s_{\alpha,\cris,x'}$. Therefore, by Lemma \ref{Kisin-1.3.11}, we have $x=x'$. 
Therefore, the normalization morphism $\nu$ is injective on $k$-points for any $k\subset\overline{\F}_p$ that contains the residue field $k_E$ of $E_p$ (\ref{fields-notations-integral-model}). On the other hand, since the source and target of $\nu$ also have the same generic fibre $\Sh_K(G,X)$, this implies that $\mathscr{S}_K^-(G,X)$ is unibranch.  
Therefore, by definition ([EGA4, Chapter IV (6.15.1)]), the local ring $U_x$ of $\mathscr{S}_K^-(G,X)$ at $x$ is unibranch. Thus by \cite[Tag0C2E]{stacks-project}, 
the complete local ring $\hat{U}_x$ only has one irreducible component. 
By Lemma \ref{Kisin-Prop-2.3.5} (resp.~Lemma \ref{irreducible-components-parahoric}), each irreducible component of the complete local ring $\hat{U}_x$ is formally smooth (resp.~normal), thus each $\hat{U}_x$ is formally smooth (resp.~normal) over $\Oo_{(v)}$. 
Therefore, $\mathscr{S}_K^-(G,X)$ is also smooth (resp.~normal) over $\Oo_{(v)}$. In particular, the normalization morphism $\nu$ is an isomorphism of schemes, and that the scheme-theoretic closure of $\Sh_K(G,X)$ inside $\mathscr{S}_{K'}(\GSp,S^{\pm})_{\Oo_{(v)}}$ is already smooth (resp.~normal). 
\end{proof}

\begin{remark}
Note that in particular Corollary \ref{coro-normalization-isomorphism} does not depend on the choice of a symplectic embedding $(G,X)\hookrightarrow(\GSp,S^{\pm})$. 
\end{remark}

\begin{numberedparagraph}
In particular, in the construction of the parahoric integral model $\mathscr{S}_{K^{\circ}}(G,X)$ in \cite[$\mathsection 4.3.6$]{Kisin-Pappas}, we can simply define it as the normalization of $\mathscr{S}_K^-(G,X)$ in $\Sh_{K^{\circ}}(G,X)$. This reduces the number of normalization steps by one. 

We briefly recall the notations, and refer the reader to \textit{loc. cit.} for the details. We fix a point $x\in\mathcal{B}(G,\Q_p)$, i.e. the Bruhat-Tits building of $G$, and let $\mathcal{G}=\mathcal{G}_x$ be the corresponding smooth $\Z_p$-group scheme whose generic fibre is $G$, and such that $\mathcal{G}^0$ is a parahoric group scheme. We write $K_p^{\circ}=\mathcal{G}^{\circ}(\Z_p)$ and $K^{\circ}=K_p^{\circ}K^p$. The parahoric integral model $\mathscr{S}_{K^{\circ}}(G,X)$ with parahoric level structure $K^{\circ}$ is a priori defined as the normalization of $\mathscr{S}_K(G,X)$ in $\Sh_{K^{\circ}}(G,X)$, but can now be simplified as simply the normalization of $\mathscr{S}_K^-(G,X)$ in $\Sh_{K^{\circ}}(G,X)$, as above. 
\end{numberedparagraph}

\begin{numberedparagraph}
The fact that the mod $p$ Hodge cycles (tensors) are determined by either its $\ell$-adic or its cristalline component (Proposition \ref{final-num-cris-triviality-motivated}) can be thought of as a ``rationality'' statement. As a historical remark, we remind the reader of the following rationality conjecture of Deligne's, which is a weakened form of the Hodge conjecture.
\begin{conj} \textit{(Deligne)}
Suppose two abelian varieties $A_1$ and $A_2$ defined over $\overline{\Q}$ have the same reduction $A$ over $\overline{\F}_p$, and suppose there are given Hodge cycles $\xi_1$ and $\xi_2$ on $A_1$ and $A_2$, respectively. Then the intersection number of the reductions $\overline{\xi}_{1,\et}$ and $\overline{\xi}_{2,\et}$ is rational. 
\end{conj}
In our specific situation for points on $\mathscr{S}_K(G,X)$, given $\mcA_{\widetilde{x}}$ and $\mcA_{\widetilde{x}'}$ equipped with Hodge cycles $(s_{\alpha,\widetilde{x}})$ and $(s_{\alpha,\widetilde{x}'})$ respectively. Suppose $s_{\alpha,\widetilde{x}}$ and $s_{\alpha',\widetilde{x}'}$ have \textit{complementary degrees} viewed as cohomology classes on powers of $\mcA_{\widetilde{x}}$ and $\mcA_{\widetilde{x}'}$ respectively. 
Then Deligne's rationality conjecture is automatic in this case: since $s_{\alpha,\widetilde{x}}$ deforms to a Hodge cycle on $\mcA_{\widetilde{x}'}$ (and likewise $s_{\alpha',\widetilde{x}'}$ deforms to a Hodge cycle on $\mcA_{\widetilde{x}}$), the intersection number $s_{\alpha,\ell,x}\cup s_{\alpha',\ell,x'}$ is rational by rationality in characteristic zero. The same is true at $p$, i.e. $s_{\alpha,\cris,x}\cup s_{\alpha,\cris,x'}$ is rational. 
\end{numberedparagraph}

\subsection{Toroidal compactifications of integral models}\label{toroidal-cpct-section}
We briefly mention one application of our main theorem to the toroidal compactifications of integral models of Hodge type constructed in \cite{Keerthi-compactification}. 
The main input is an analysis on the boundary components from \cite{Lan-immersion}, where the level structure is taken to be \textit{neat}, i.e. sufficiently small. We remark that while Lan's result is conditional on the existence of an embedding on the open part, our result is unconditional (since we proved the embedding on the open part). 
\begin{numberedparagraph}
We adopt the notations \textit{loc. cit.}. Let $\mathscr{S}_K^{\Sigma}(G,X)=\coprod \mathcal{Z}_{[\sigma]}$ be a stratification into locally closed subschemes, where $\sigma\in \Sigma_{\mathcal{Z}}^+$. Let $\mathscr{S}_{K'}^{\Sigma'}(\GSp,S^{\pm})=\coprod \mathcal{Z}'_{[\tau]}$ be a stratification, where $\tau\in {\Sigma^{'+}_{\mathcal{Z}'}}$. Note that the Hodge morphism $\mathscr{S}_K(G,X)\to\mathscr{S}_{K'}(\GSp,S^{\pm})$ extends uniquely to a morphism $\mathscr{S}_K^{\Sigma}(G,X)\to \mathscr{S}_{K'}^{\Sigma'}(\GSp,S^{\pm})$ between their toroidal compactifications. In particular, the following diagram commutes:
\begin{equation}\label{compactification-functoriality}
\begin{tikzcd}\mathscr{S}_K^{\Sigma}(G,X)\arrow[]{r}{}&\mathscr{S}_{K'}^{\Sigma'}(\GSp,S^{\pm})\\
\mathscr{S}_K(G,X)\arrow[hook]{u}{}\arrow[]{r}{}&\mathscr{S}_{K'}(\GSp,S^{\pm})\arrow[hook]{u}{}\end{tikzcd}
\end{equation}
\end{numberedparagraph}

\begin{Coro}\label{toroidal-cpct-corollary} 
Let $(G,X)$ be a Shimura datum of Hodge type. 
For each $K\subset G(\A_f)$ sufficiently small\footnote{When the level structure is parahoric, one needs to also impose the assumptions as in \ref{parahoric-analogue-intro}}, there exist collections $\Sigma$ and $\Sigma'$ of cone decompositions, and $K'\subset\GSp(\A_f)$, such that we have a closed embedding
\begin{equation}\label{toroidal-cpct-embedding}
\mathscr{S}_K^{\Sigma}(G,X)\hookrightarrow\mathscr{S}_{K'}^{\Sigma'}(\GSp,S^{\pm})
\end{equation}
extending the Hodge embedding of integral models. \\
In particular, the normalization step is redundant, and $\mathscr{S}_K^{\Sigma}(G,X)$ can be constructed by simply taking the closure of $\Sh_K(G,X)$ inside $\mathscr{S}_{K'}^{\Sigma'}(\GSp,S^{\pm})$. 
\end{Coro}
\begin{proof}
The result follows immediately by combining \cite[Thm 2.2]{Lan-immersion} with our main Theorem \ref{Main-Theorem-intro}. For the reader's convenience, we sketch the argument. 

Let $x'\in \mathscr{S}_{K'}^{\Sigma'}(\GSp,S^{\pm})$ be any point that lies in the stratum $\mathcal{Z}'_{[\tau]}$. Then \'etale locally at $x'$, \ref{compactification-functoriality} becomes
\begin{equation}\label{compactification-functoriality-local}
\begin{tikzcd}
    \coprod\limits_j \big(E_{\mathcal{Z},j}(\sigma_j)\times_{\Spec\Z}C_{\mathcal{Z},j}\big)\arrow[]{r}{}& E_{\mathcal{Z}'}(\tau)\times_{\Spec\Z}C_{\mathcal{Z}'}\\
     \coprod\limits_j \big(E_{\mathcal{Z},j}\times_{\Spec\Z}C_{\mathcal{Z},j}\big)\arrow[]{r}{}\arrow[hook]{u}{}&E_{\mathcal{Z}'}\times_{\Spec\Z}C_{\mathcal{Z}'}\arrow[hook]{u}{}
     \end{tikzcd}
\end{equation}
More precisely, there exists an \'etale neighborhood $\overline{U}'\to \mathscr{S}_{K'}^{\Sigma'}(\GSp,S^{\pm})$ of $x'$ and an \'etale morphism $\overline{U}'\to E_{\mathcal{Z}'}(\tau)\times_{\Spec\Z}C_{\mathcal{Z}'}$ which pullback via the Hodge morphism to \'etale morphisms $\overline{U}\to\mathscr{S}_K^{\Sigma}(G,X)$ and $\overline{U}\to \coprod\limits_j (E_{\mathcal{Z},j}(\sigma_j)\times_{\Spec\Z}C_{\mathcal{Z},j})$. 

By \cite[Prop 4.9]{Lan-immersion}, there exist ``strictly compatible collections'' (see \cite[Definition 4.6]{Lan-immersion}) $\Sigma$ and $\Sigma'$ of cone decompositions with respect to the Hodge morphism $\Phi$. By Theorem \ref{Main-Theorem-intro} (resp.~Theorem \ref{parahoric-analogue-intro}), the Hodge morphism 
\[\Phi: \mathscr{S}_K(G,X)\to\mathscr{S}_{K'}(\GSp,S^{\pm})\]
is a closed embedding. Therefore, diagram \ref{compactification-functoriality-local} gives a closed embedding
\begin{equation}\label{lower-horizontal-compactification-functoriality-local}
\coprod\limits_j (E_{\mathcal{Z},j}\times_{\Spec\Z}C_{\mathcal{Z},j})\hookrightarrow E_{\mathcal{Z}'}\times_{\Spec\Z}C_{\mathcal{Z}'}
\end{equation}
over certain \'etale neighborhood $\overline{U}'$ of $x'$ (see \cite[3.11]{Lan-immersion}). This then implies that 
\begin{equation}\label{embedding-on-C-toroidal}
C_{\mathcal{Z},j}\to C_{\mathcal{Z}'}
\end{equation}
are closed embeddings over the image of $\overline{U}'$ in $C_{\mathcal{Z}'}$ for all $j$. 
On the other hand, one can check directly from the  construction, as in \cite[Lemma 4.3]{Lan-immersion}, that \begin{equation}\label{embedding-on-Esigma-toroidal}
E_{\mathcal{Z},j}(\sigma_j)\hookrightarrow E_{\mathcal{Z}'}(\tau)
\end{equation}
is a closed embedding. Therefore for each $j$, by \ref{embedding-on-C-toroidal} and \ref{embedding-on-Esigma-toroidal} the map 
\[E_{\mathcal{Z},j}(\sigma_j)\times_{\Spec\Z}C_{\mathcal{Z},j}\hookrightarrow E_{\mathcal{Z}'}(\tau)\times_{\Spec\Z}C_{\mathcal{Z}'}\]
is a closed embedding over the image of $\overline{U}'$. To show that 
\begin{equation}\label{upper-horizontal-compactification-functoriality-local}
\coprod\limits_j\big(E_{\mathcal{Z},j}(\sigma_j)\times_{\Spec\Z}C_{\mathcal{Z},j}\big)\to E_{\mathcal{Z}'}(\tau)\times_{\Spec\Z}C_{\mathcal{Z}'}
\end{equation}
is a closed embedding, it then suffices to show that any point $x'$ in the image of \ref{upper-horizontal-compactification-functoriality-local} can only come from one of term indexed by $j$ on the left-hand-side. We show this by contradiction. Suppose there are $j_1\neq j_2$ and points $y_1$ (resp. $y_2$) of $E_{\mathcal{Z},j_1}(\sigma_{j_1})\times_{\Spec\Z}C_{\mathcal{Z},j_1}$ (resp. $E_{\mathcal{Z},j_2}(\sigma_{j_2})\times_{\Spec\Z}C_{\mathcal{Z},j_2}$) that map to the point $x'$ of $E_{\mathcal{Z}'}(\tau)\times_{\Spec\Z}C_{\mathcal{Z}'}$ under \ref{upper-horizontal-compactification-functoriality-local}. Then $x', y_1$ and $y_2$ have the same image $z'\in C_{\mathcal{Z}'}$. On the other hand, $z'$ is also in the image of closed embeddings $C_{\mathcal{Z},j_1}\hookrightarrow C_{\mathcal{Z}'}$ and $C_{\mathcal{Z},j_2}\hookrightarrow C_{\mathcal{Z}'}$. By considering the pullbacks to $z'$ of the closed embeddings $ E_{\mathcal{Z},j_1}(\sigma_{j_1})\to E_{\mathcal{Z}'}(\tau)$ and $E_{\mathcal{Z},j_1}(\sigma_{j_1})\to E_{\mathcal{Z}'}(\tau)$, we obtain closed embeddings $\alpha_{j_1}: E_{\mathcal{Z},j_1}(\sigma_{j_1})_{z'}\to E_{\mathcal{Z}'}(\tau)_{z'}$ and $\alpha_{j_2}: E_{\mathcal{Z},j_1}(\sigma_{j_1})_{z'}\to E_{\mathcal{Z}'}(\tau)_{z'}$ which have overlapping images. One then argues as in \cite[4.4, 4.8]{Lan-immersion} to conclude that $(E_{\mathcal{Z},j_1})_{z'}\to (E_{\mathcal{Z}'})_{z'}$ and $(E_{\mathcal{Z},j_2})_{z'}\to (E_{\mathcal{Z}'})_{z'}$ would be closed embeddings with overlapping images, but then this contradicts the Hodge embedding in Theorem \ref{Main-Theorem-intro} (resp.~Theorem \ref{parahoric-analogue-intro}), or more specifically its local consequence \ref{lower-horizontal-compactification-functoriality-local}. 
\end{proof}

\begin{remark}
Note that our current proof for Corollary \ref{toroidal-cpct-corollary} only gives us the \textit{existence} of cone decompositions $\Sigma,\Sigma'$ that produce closed embeddings of the form \ref{toroidal-cpct-embedding}. Our statement does not imply that for any cone decomposition $\Sigma'$ on $\mathscr{S}_{K'}(\GSp,S^{\pm})$, there is an induced $\Sigma$ on $\mathscr{S}_K(G,X)$ such that $(\Sigma,\Sigma')$ gives the desired closed embedding \ref{toroidal-cpct-embedding}. It is unclear to the author at the moment whether \ref{toroidal-cpct-embedding} can be constructed for any $\Sigma'$. 
\end{remark}

\bibliographystyle{amsalpha}
\bibliography{bibfile}

\providecommand{\bysame}{\leavevmode\hbox to3em{\hrulefill}\thinspace}
\providecommand{\MR}{\relax\ifhmode\unskip\space\fi MR }
\providecommand{\MRhref}[2]{%
  \href{http://www.ams.org/mathscinet-getitem?mr=#1}{#2}
}
\providecommand{\href}[2]{#2}
\begin{thebibliography}{{Sta}20}

\bibitem[dJ95]{deJong-formal-rigid}
A.~J. de~Jong, \emph{Crystalline {D}ieudonn\'{e} module theory via formal and
  rigid geometry}, Inst. Hautes \'{E}tudes Sci. Publ. Math. (1995), no.~82,
  5--96 (1996). \MR{1383213}

\bibitem[Kis10]{Kisin-integral-model}
Mark Kisin, \emph{Integral models for {S}himura varieties of abelian type}, J.
  Amer. Math. Soc. \textbf{23} (2010), no.~4, 967--1012.

\bibitem[Kis17]{Kisin-mod-p-points}
\bysame, \emph{{${\rm mod}\,p$} points on {S}himura varieties of abelian type},
  J. Amer. Math. Soc. \textbf{30} (2017), no.~3, 819--914.

\bibitem[KMP16]{2-adic-integral-model}
Wansu Kim and Keerthi Madapusi~Pera, \emph{2-adic integral canonical models},
  Forum Math. Sigma \textbf{4} (2016), e28, 34. \MR{3569319}

\bibitem[Kot92]{Kottwitz}
Robert~E. Kottwitz, \emph{Points on some shimura varieties over finite fields},
  Journal of the American Mathematical Society \textbf{5} (1992), no.~2,
  373--444 (eng).

\bibitem[KP18]{Kisin-Pappas}
M.~Kisin and G.~Pappas, \emph{Integral models of {S}himura varieties with
  parahoric level structure}, Publ. Math. Inst. Hautes \'{E}tudes Sci.
  \textbf{128} (2018), 121--218. \MR{3905466}

\bibitem[Lan13]{Lan-thesis}
Kai-Wen Lan, \emph{Arithmetic compactifications of {PEL}-type {S}himura
  varieties}, London Mathematical Society Monographs Series, vol.~36, Princeton
  University Press, Princeton, NJ, 2013. \MR{3186092}

\bibitem[Lan19]{Lan-immersion}
\bysame, \emph{Closed immersions of toroidal compactifications of shimura
  varieties}, 2019.

\bibitem[MP19]{Keerthi-compactification}
Keerthi Madapusi~Pera, \emph{Toroidal compactifications of integral models of
  {S}himura varieties of {H}odge type}, Ann. Sci. \'{E}c. Norm. Sup\'{e}r. (4)
  \textbf{52} (2019), no.~2, 393--514. \MR{3948111}

\bibitem[PZ13]{Pappas-Zhu}
G.~Pappas and X.~Zhu, \emph{Local models of {S}himura varieties and a
  conjecture of {K}ottwitz}, Invent. Math. \textbf{194} (2013), no.~1,
  147--254. \MR{3103258}

\bibitem[RZ96]{Rapoport-Zink}
M.~Rapoport and Th. Zink, \emph{Period spaces for {$p$}-divisible groups},
  Annals of Mathematics Studies, vol. 141, Princeton University Press,
  Princeton, NJ, 1996. \MR{1393439}

\bibitem[{Sta}20]{stacks-project}
The {Stacks Project Authors}, \emph{\textit{Stacks Project}},
  \url{https://stacks.math.columbia.edu}, 2020.

\bibitem[Tat66]{Tate-endo-ab-var-fin-fields}
John Tate, \emph{Endomorphisms of abelian varieties over finite fields},
  Invent. Math. \textbf{2} (1966), 134--144.

\bibitem[Xu21]{PEL-embedding}
Yujie Xu, \emph{\textit{On the Hodge embedding for PEL type integral models of
  Shimura varieties}}, \url{http://arxiv.org/abs/2111.04209}, 2021.

\bibitem[Zho17]{Rong-mod-p}
Rong Zhou, \emph{Mod-{$p$} isogeny classes on shimura varieties with parahoric
  level structure}, 07 2017.

\end{thebibliography}

\end{document}